\documentclass[a4paper,12pt]{amsart}

\usepackage{amssymb}
\usepackage{epsfig}
\usepackage{amsfonts}
\usepackage{amsmath}
\usepackage{euscript}
\usepackage{amscd}
\usepackage{amsthm}
\usepackage{comment}
\usepackage{bbm}
\usepackage[all,cmtip]{xy}
\DeclareMathAlphabet{\mathpzc}{OT1}{pzc}{m}{it}

\usepackage{soul}
\setstcolor{red}

\makeatletter
\@namedef{subjclassname@2020}{%
  \textup{2020} Mathematics Subject Classification}
\makeatother

\usepackage[colorinlistoftodos,textwidth=2cm,textsize=small]{todonotes}
\usepackage{verbatim}
\usepackage[pagebackref,hypertexnames=false, colorlinks, citecolor=red,linkcolor=blue, urlcolor=red]{hyperref}
\usepackage{mathrsfs}
\usepackage{tikz}

\DeclareSymbolFont{SY}{U}{psy}{m}{n}
\DeclareMathSymbol{\emptyset}{\mathord}{SY}{'306}

\theoremstyle{plain}

\newtheorem{thm}{Theorem}[section]
\newtheorem{cor}[thm]{Corollary}
\newtheorem{lem}[thm]{Lemma}
\newtheorem{prop}[thm]{Proposition}
\theoremstyle{definition}

\newtheorem{rem}[thm]{Remark}

\numberwithin{equation}{section}

\def\C{{\mathbb C}}
\def\D{{\mathbb D}}

\def\inp#1,#2{\left\langle{#1},{#2}\right\rangle}

\def\l{\lambda}

\def\beq{\begin{eqnarray}}
\def\eeq{\end{eqnarray}}
\def\beqa{\begin{eqnarray*}}
\def\eeqa{\end{eqnarray*}}

\def\<{\langle}
\def\>{\rangle}

\def\bz{\boldsymbol{z}}
\def\bw{\boldsymbol{w}}

\newcommand{\be}{\begin{equation}}
\newcommand{\ee}{\end{equation}}
\newcommand{\bea}{\begin{eqnarray}}
\newcommand{\eea}{\end{eqnarray}}
\newcommand{\Bea}{\begin{eqnarray*}}
\newcommand{\Eea}{\end{eqnarray*}}

\newcounter{cnt1}
\newcounter{cnt2}
\newcounter{cnt3}
\newcommand{\blr}{\begin{list}{$($\roman{cnt1}$)$}
 {\usecounter{cnt1} \setlength{\topsep}{0pt}
 \setlength{\itemsep}{0pt}}}
\newcommand{\bla}{\begin{list}{$($\alph{cnt2}$)$}
 {\usecounter{cnt2} \setlength{\topsep}{0pt}
 \setlength{\itemsep}{0pt}}}
\newcommand{\bln}{\begin{list}{$($\arabic{cnt3}$)$}
 {\usecounter{cnt3} \setlength{\topsep}{0pt}
 \setlength{\itemsep}{0pt}}}
\newcommand{\el}{\end{list}}
\newcommand{\overbar}[1]{\mkern 1.5mu\overline{\mkern-1.5mu#1\mkern-1.5mu}\mkern 1.5mu}

\DeclareMathOperator{\aut}{Aut}

\newcounter{defcounter}
\begin{document}
\title[M\"ob - homogeneous]{Representations of the M\"{o}bius group and pairs of homogeneous operators in the Cowen-Douglas class}

\author[J. Das]{Jyotirmay Das}
\address[J. Das]{Institute for Advancing Intelligence\\ TCG Centres for Research and Education in Science and Technology \\ Kolkata }
\address{Department of Mathematics\\ National Institute of Technology Durgapur\\ India }
\email{iamjyotirmay1999@gmail.com}

\author[S. Hazra]{Somnath Hazra}

\address[S. Hazra]{Institute for Advancing Intelligence\\ TCG Centres for Research and Education in Science and Technology \\ Kolkata }
\address{Academy of Scientific and Innovative Research (AcSIR)\\ Ghaziabad- 201002\\ India.}
\email{somnath.hazra@tcgcrest.org}

\thanks{The research of the first named author was supported by TCG CREST Ph.D Fellowship.}

\begin{abstract} 
Let M\"ob be the biholomorphic automorphism group of the unit disc of the complex plane, $\mathcal{H}$ be a complex separable Hilbert space and $\mathcal{U}(\mathcal{H})$ be the group of all unitary operators. Suppose $\mathcal{H}$ is a reproducing kernel Hilbert space consisting of holomorphic functions over the poly-disc $\mathbb D^n$ and contains all the polynomials. If $\pi : \mbox{M\"ob} \to \mathcal{U}(\mathcal{H})$ is a multiplier representation, then we prove that there exist $\lambda_1, \lambda_2, \ldots, \lambda_n > 0$ such that $\pi$ is unitarily equivalent to $(\otimes_{i=1}^{n} D_{\lambda_i}^+)|_{\mbox{\tiny{M\"ob}}}$, where each $D_{\lambda_i}^+$ is a holomorphic discrete series representation of M\"ob. As an application, we prove that if $(T_1, T_2)$ is a M\"ob - homogeneous pair in the Cowen - Douglas class of rank $1$ over the bi-disc, then each $T_i$ posses an upper triangular form with respect to a decomposition of the Hilbert space. In this upper triangular form of each $T_i$, the diagonal operators are identified. We also prove that if $\mathcal{H}$ consists of symmetric (resp. anti-symmetric) holomorphic functions over $\mathbb D^2$ and contains all the symmetric (resp. anti-symmetric) polynomials, then there exists $\lambda > 0$ such that $\pi \cong \oplus_{m = 0}^\infty D^+_{\lambda + 4m}$ (resp. $\pi \cong \oplus_{m=0}^\infty D^+_{\lambda + 4m + 2}$). 

\end{abstract}

\maketitle

\section{Introduction}
Let $G$ be a locally compact second countable group, $\mathcal{H}$ be a complex separable Hilbert space and $\mathcal{U}(\mathcal{H})$ be the group of all unitary operators on $\mathcal{H}$. A projective unitary representation of $G$ on $\mathcal{H}$ is a Borel map $\pi : G \to \mathcal{U}(\mathcal{H})$ such that 
$$\pi(1) = I,\,\,\pi(gh) = m(g, h) \pi(g) \pi(h),\,\,g, h \in G,$$
where $m : G \times G \to \mathbb T$ is a Borel map ($\mathbb T$ is the circle group) \cite[p. 248]{VAR}. Throughout this article, a projective unitary representation is called a projective representation. The function $m$ associated with a projective representation $\pi$ is called the multiplier of $\pi$. The multiplier $m$ must satisfy the following equations
$$m(g, 1) = m(1, g) = 1,\,\,m(g_1, g_2)m(g_1 g_2, g_3) = m(g_1, g_2g_3) m(g_2, g_3),$$
for all $g, g_1, g_2, g_3 \in G$ \cite[p. 247]{VAR}. Two projective representations $\pi_i : G \to \mathcal{U}(\mathcal{H}_i)$, $i = 1, 2$ are said to be equivalent, denoted by $\pi_1 \cong \pi_2$, if there exists a Borel map $f : G \to \mathbb T$ and a unitary operator $U : \mathcal{H}_1 \to \mathcal{H}_2$ such that 
$$\pi_1(g) = f(g) U^* \pi_2(g) U$$
holds for every $g \in G$ \cite[Definition 3.2]{THS}.

A closed subspace $\mathcal{A}$ of $\mathcal{H}$ is said to be an invariant subspace of a projective representation $\pi : G \to \mathcal{U}(\mathcal{H})$, if $\pi(g) \mathcal{A} \subseteq \mathcal{A}$ for every $g \in G$. It is straightforward to verify that if $\mathcal{A}$ is an invariant subspace of $\pi$, then $\mathcal{A}$ is also a reducing subspace for every $\pi(g)$. A projective representation is said to be irreducible, if it does not have any reducing subspace. Given a projective representation $\pi$, it is natural to ask if there exists a decomposition of $\mathcal{H}$ into a direct sum of irreducible subspaces of $\pi$. 

If $\Omega$ is a bounded symmetric domain and $G$ is the biholomorphic automorphism group of $\Omega$, a decomposition of tensor product of two holomorphic discrete series representations of $G$ is obtained in \cite{TPHR}.
Let M\"ob be the group of all biholomorphic automorphisms of the unit disc $\mathbb D := \{z \in \mathbb C : |z| < 1\}$ of the complex plane $\mathbb C$.
In this article, we describe every multiplier representation of M\"ob on reproducing kernel Hilbert spaces consisting of holomorphic functions of $\mathbb D^2$. This, in particular, produces a decomposition of the tensor product of two holomorphic discrete series representations of M\"ob.

Let $X$ be a topological space and let $G$ have an action on $X$. Suppose $\mathcal{H}$ is a Hilbert space and $\mathcal{M}$ is a dense subspace of $\mathcal{H}$ which consists of functions on $X$. A projective representation $\pi$ of $G$ on $\mathcal{H}$ is said to be a multiplier representation if $\mathcal{M}$ is invariant under $\pi$ and
\begin{equation}\label{eqn: in1}
\left(\pi(g^{-1})f\right)(x) = c(g, x) f(g \cdot x),\,\,g \in G,\,\,f \in \mathcal{M},\,\,x \in X,
\end{equation}
where $c : G \times X \to \mathbb C$ is a non-vanishing measurable function \cite[p. 299]{THS}. Note that the Equation \eqref{eqn: in1} defines a projective representation with an associated multiplier $m$ if and only if $c$ satisfies the following equation
\begin{equation}\label{eqn: in2}
c(gh, x) = m(h^{-1}, g^{-1})c(g, h \cdot x) c(h, x)  
\end{equation}
for every $g, h \in G$ and $x \in X$. A non-vanishing function $c : G \times X \to \mathbb C$ satisfying the Equation \eqref{eqn: in2} is called a projective cocycle. For simplicity, throughout this article, a projective cocycle is called a cocycle.


For $\lambda > 0$, let $\mathbb A^{(\lambda)}(\mathbb D)$ be the reproducing kernel Hilbert space with the reproducing kernel $B^{(\lambda)}(z, w) = (1 - z \bar{w})^{- \lambda}$, $z, w \in \mathbb D$. 
Note that for each $\varphi \in$ M\"ob, $\varphi'(z)$ lies in the open disc of radius $1$ centered at $1$. Therefore, there exists a branch of logarithm such that $z(\in \D) \to \log \varphi'(z)$ is analytic. For the rest of the paper, we fix such a branch of logarithm for each $\varphi \in$ M\"ob with the convention that $\log \varphi' \equiv 0$ if $\varphi$ is the identity map.
The cocycle $c^{(\lambda)} : \mbox{M\"ob} \times \mathbb D \to \mathbb C$, given by $c^{(\lambda)}(\varphi, z) = \varphi'(z)^{\frac{\lambda}{2}}$, defines a projective representation $D_\lambda^+$ of M\"ob on $\mathbb A^{(\lambda)}(\mathbb D)$ by the Equation \eqref{eqn: in1}. These are called the holomorphic discrete series representations of M\"ob \cite[List 3.1(1)]{THS}.

Let M\"ob $\times$ M\"ob denote the direct product of two copies of M\"ob. For any $\lambda, \mu > 0$, $D_\lambda^+ \otimes D_\mu^+$ naturally defines a representation of M\"ob $\times$ M\"ob on the Hilbert space $\mathbb A^{(\lambda)}(\mathbb D) \otimes \mathbb A^{(\mu)}(\mathbb D)$. The tensor product $D_\lambda^+ \otimes D_\mu^+$ is an irreducible representation of M\"ob $\times$ M\"ob on $\mathbb A^{(\lambda)}(\mathbb D) \otimes \mathbb A^{(\mu)}(\mathbb D)$, that is, there does not exist any proper closed subspace $\mathcal{M}$ of $\mathbb A^{(\lambda)}(\mathbb D) \otimes \mathbb A^{(\mu)}(\mathbb D)$ such that $\mathcal{M}$ is a reducing subspace of $D_\lambda^+(\varphi_1) \otimes D_\mu^+(\varphi_2)$ for every $\varphi_1, \varphi_2$ in M\"ob. Since the diagonal subgroup $\{(\varphi, \varphi) : \varphi \in \mbox{M\"ob}\}$ of M\"ob $\times$ M\"ob is naturally isomorphic to M\"ob, we denote it by M\"ob. Note that the restriction of the representation $D_\lambda^+ \otimes D_\mu^+$ to M\"ob is not irreducible. In fact, $D_\lambda^+ \otimes D_\mu^+|_{\mbox{\tiny{M\"ob}}} \cong \displaystyle \oplus_{n=0}^\infty D_{\lambda + \mu + 2n}^+$. This is the Clebsh-Gordan formula.

Note that the group M\"ob has a natural action on the bi-disc $\D^2$ via the map $(\varphi, (z_1, z_2)) \to (\varphi(z_1), \varphi(z_2))$, $\varphi \in$ M\"ob, $(z_1, z_2) \in \D^2$. Suppose $\pi$ is a multiplier representation of M\"ob on a reproducing kernel Hilbert space $\mathcal{H}$. It is  proved that if $\mathcal{H}$ consists of holomorphic functions over $\D^2$ and contains all the polynomials $\C[\bz]$, then $\mathcal{H}$ is decomposed into orthogonal direct sum $\mathcal{H} = \oplus_{n = 0}^\infty \mathcal{S}_n$ of irreducible subspaces of $\pi$. It is also observed that there exists $\eta > 0$ such that $\pi|_{\mathcal{S}_n} \cong D_{\eta+2n}^+$ for every $n \geq 0$ and therefore, $\pi \cong \displaystyle \oplus_{n=0}^\infty D_{\eta+2n}^+$. The Clebsh-Gordan formula is an immediate consequence of this decomposition. It is also observed that if $\mathcal{H}$ consists of symmetric (resp. anti-symmetric) holomorphic functions over $\D^2$ and contains all the symmetric (resp. anti-symmetric) polynomials, then there exists $\eta > 0$ such that $\pi \cong \oplus_{n=0}^\infty D^+_{\eta + 4n}$ (resp. $\pi \cong \oplus_{n=0}^\infty D^+_{\eta + 4n + 2}$). In Section \ref{section 2}, Theorem \ref{thm3} verifies all these assertions. As an immediate consequence of the Theorem \ref{thm3}(i) and the Clebsh-Gordan formula, we observe that every multiplier representation $\pi$ on a reproducing kernel Hilbert $\mathcal{H} \subseteq \mbox{Hol}(\D^2, \C)$ containing $\C[\bz]$ is equivalent to $D^+_{\l_1} \otimes D^+_{\l_2}|_{\mbox{\tiny{M\"ob}}}$ for some $\l_1, \l_2 > 0$.

For $\lambda > 0$, let $\mathbb A^{(\lambda)}_{\mbox{\tiny{sym}}}(\D^2)$ and $\mathbb A^{(\lambda)}_{\mbox{\tiny{anti}}}(\D^2)$ denote the subspaces of $\mathbb A^{(\lambda)}(\mathbb D) \otimes \mathbb A^{(\lambda)}(\mathbb D)$ consisting of symmetric functions and anti-symmetric functions, respectively. The subspaces $\mathbb A^{(\lambda)}_{\mbox{\tiny{sym}}}(\D^2)$ and $\mathbb A^{(\lambda)}_{\mbox{\tiny{anti}}}(\D^2)$ are invariant under the representation $D^+_\lambda \otimes D_\lambda^+|_{\mbox{\tiny{M\"ob}}}$. Finally, it is observed as a corollary of Theorem \ref{thm3} in Section \ref{section 2} that the representation $D^+_\lambda \otimes D_\lambda^+ : \mbox{M\"ob} \to \mathcal{U}\left(A^{(\lambda)}_{\mbox{\tiny{sym}}}(\D^2)\right)$ is unitarily equivalent to the representation $\displaystyle \oplus_{n=0}^\infty D_{2\lambda + 4n}^+$ and consequently, the representation $D^+_\lambda \otimes D_\lambda^+ : \mbox{M\"ob} \to \mathcal{U}\left(A^{(\lambda)}_{\mbox{\tiny{anti}}}(\D^2)\right)$ is unitarily equivalent to the representation $\displaystyle \oplus_{n=0}^\infty D_{2\lambda + 4n + 2}^+$

It is then verified in the Section \ref{section 2} that the Theorem \ref{thm3} has a natural generalization to the poly-disc case. The group M\"ob also has a natural action to the poly-disc $\D^n$ via the map 
$$(\varphi, (z_1, \ldots, z_n)) \to (\varphi(z_1), \ldots, \varphi(z_n)),$$
$\varphi \in$ M\"ob and $(z_1, \ldots, z_n) \in \mathbb D^n$. If $\pi : \mbox{M\"ob} \to \mathcal{U}(\mathcal{H})$ is a multiplier representation where $\mathcal{H} \subseteq \mbox{Hol}(\mathbb D^n, \mathbb C)$ is a reproducing kernel Hilbert space and contains the set of all polynomials $\C[\bz]$, then there exist $\l_1, \ldots, \l_n > 0$ such that $\pi \cong D^+_{\l_1} \otimes \cdots \otimes D^+_{\l_n}|_{\mbox{\tiny{M\"ob}}}$. This is the Theorem \ref{main thm}.

A bounded linear operator $T$ on a complex separable Hilbert space $\mathcal{H}$ is said to be homogeneous if the spectrum of $T$ is contained in $\overline{\D}$ and, for every $\varphi \in$ M\"ob, $\varphi(T)$ is unitarily equivalent to $T$. 
The class of homogeneous operators has been studied in a number of articles \cite{HOPRMGS, THS, HHHVB1, HOHS, ACHOCD, OICHO, HVBIOSD, HOJCS}.
The notion of a homogeneous operator has a natural generalization to a commuting tuple of operators. 
Let $\Omega$ be a bounded symmetric domain in $\C^n$ and $G$ be a subgroup of the biholomorphic automorphism group $\aut(\Omega)$ of $\Omega$. A commuting $n$ - tuple of operators $(T_1, T_2,\ldots ,T_n)$ on a complex separable Hilbert space is said to be homogeneous with respect to $G$ (or $G$-homogeneous) if the Taylor joint spectrum of $(T_1, T_2,\ldots , T_n)$ lies in $\overline{\Omega}$ and $g(T_1, T_2,\ldots , T_n)$ is unitarily equivalent with $(T_1, T_2,\ldots ,T_n)$ for all $g \in G$ \cite[Definition 1.1]{HTOHDSR}. In this article, we consider the case when $\Omega = \mathbb D^2$ and $G$ is the diagonal subgroup $\{(\varphi, \varphi) : \varphi \in \mbox{M\"ob}\}$ of M\"ob $\times$ M\"ob. We refer to a pair of operators that is homogeneous with respect to the group $\{(\varphi, \varphi): \varphi \in \mbox{M\"ob}\}$ as a M\"ob - homogeneous pair, since the subgroup $\{(\varphi, \varphi): \varphi \in \mbox{M\"ob}\}$ is naturally identified with M\"ob.

For every $\lambda, \mu > 0$, the pairs of multiplication operators by the coordinate functions on $\mathbb A^{(\lambda)}(\mathbb D) \otimes \mathbb A^{(\mu)}(\mathbb D)$ are examples of M\"ob - homogeneous pairs. However, these are also M\"ob $\times$ M\"ob - homogeneous pairs (cf. \cite[Theorem 3.1]{HHHVB1}). This naturally raises the question of finding M\"ob - homogeneous pairs of operators that are not M\"ob $\times$ M\"ob - homogeneous and classifying all the M\"ob - homogeneous pairs.
 
Let $(T_1, T_2)$ be a M\"ob - homogeneous pair of operators on a complex separable Hilbert space $\mathcal{H}$ in the Cowen - Douglas class $B_1(\D^2)$ of rank $1$ over $\D^2$. Note that every operator pair in $B_1(\mathbb D^2)$ is realized as an adjoint of the pair of multiplication by the coordinate functions on a reproducing kernel Hilbert space consisting of holomorphic functions of $\D^2$ (cf. \cite{GBKCD}). Therefore, we assume that $(T_1, T_2)$ is $(M_{z_1}^*, M_{z_2}^*)$ on a reproducing kernel Hilbert space $(\mathcal{H}, K) \subseteq \mbox{Hol}(\mathbb D^2, \mathbb C)$ where $K : \mathbb D^2 \times \mathbb D^2 \to \mathbb C$ is a reproducing kernel. We first obsreve that there exists a multiplier representation $\pi : \mbox{M\"ob} \to \mathcal{U}(\mathcal{H})$ such that $\varphi (M_{z_1}, M_{z_2}) = (\varphi(M_{z_1}), \varphi(M_{z_2})) = \pi(\varphi)^* (M_{z_1}, M_{z_2}) \pi(\varphi)$ for every $\varphi$ in M\"ob. Now, Theorem \ref{thm3}(i) yields that $\mathcal{H}$ is decomposed into orthogonal direct sum $\mathcal{H} = \oplus_{n = 0}^\infty \mathcal{S}_n$ of irreducible subspaces of $\pi$. It is then verified that $\displaystyle \oplus_{i = 0}^m \mathcal{S}_i$ is an invariant subspace of each $M_{z_i}^*$, $i = 1, 2,$ for every $m \geq 0$. Let $P_n$ denote the orthogonal projection onto the subspace $\mathcal{S}_n$ for every $n \geq 0$. Additionally, for any $n \geq 0$ and $i = 1, 2$, it is demonstrated that $P_n M_{z_i}^*|_{\mathcal{S}_n}$ is unitarily equivalent to the adjoint of the multiplication operator ${M^{(\lambda + 2n)}}$ by the coordinate function on $\mathbb A^{(\lambda + 2n)}(\D)$ for some $\lambda > 0$. All these assertions are verified in Section \ref{section 3}.



\section{Representations of M\"ob}\label{section 2}
Let $\mathcal{H}$ be a reproducing kernel Hilbert space consisting of holomorphic functions on $\D^n$ with reproducing kernel $K : \D^n \times \D^n \to \C$ and $\pi : \mbox{M\"ob} \to \mathcal{U}(\mathcal{H})$ be a non-trivial multiplier representation, defined by
\begin{equation}\label{eq:2}
\left(\pi(\varphi^{-1})f\right)(\bz) = c(\varphi, \bz)f(\varphi(\bz)),\,\,\varphi \in \mbox{M\"ob},\,\,f \in \mathcal{H},\,\,\bz \in \D^n.
\end{equation} 
Note that the cocycle $c$ is a Borel map in first variable and a holomorphic map in second variable. In this section, we find a decomposition of $\mathcal{H}$ into direct sum of irreducible subspaces of $\pi$ and identify the restriction of $\pi$ to each of these subspaces. This section is divided into two subsections. In the first subsection, we consider the case when $n = 2$ and in the later subsection, we consider $n$ to be arbitrary.


\subsection{The case when $n = 2$}
To describe the representation $\pi$ in this case, consider the set $\Delta := \{(z,z) : z \in \D\}$. For $m \geq 0$, let 
\begin{equation}\label{eq:14}
\mathcal{M}_m := \left\lbrace f \in \mathcal{H} : \partial_{1}^i f|_{\Delta} = 0,\,\,0 \leq i \leq m \right\rbrace.
\end{equation}
Note that each $\mathcal{M}_m$ is a closed subspace and $\mathcal{M}_m \subseteq \mathcal{M}_l$ if $m \geq l$. Consider the subspaces $\mathcal{S}_0 = \mathcal{H} \ominus \mathcal{M}_0$ and $\mathcal{S}_m = \mathcal{M}_{m-1} \ominus \mathcal{M}_m$ for $m \geq 1$. Then, we have $\mathcal{H} = \displaystyle \oplus_{m = 0}^\infty \mathcal{S}_m$.
\begin{prop}\label{prop1}
(i) For each $m \geq 0$, the subspace $\mathcal{S}_m$ is a reducing subspace of the representation $\pi$.

(ii) For every $\varphi \in$ M\"ob and $z \in \mathbb D$, $c(\varphi, (z,z)) \neq 0$.
\end{prop}

\begin{proof}
(i) Since $\pi$ is a projective unitary representation, it follows that a subspace is an invariant subspace of $\pi$ if and only if it is a reducing subspace of $\pi$. Thus, it is enough to prove that $\mathcal{M}_m$ is an invariant subspace of $\pi$ for every $m \geq 0$. Let $f \in \mathcal{M}_m$. For $\varphi \in \mbox{M\"ob}$ and $0 \leq i \leq m$, we have
\begin{equation}\label{eq:4}
 \partial_1^i \left(\pi(\varphi^{-1})f\right)(\bz) = \displaystyle \sum_{j = 0}^i \binom{i}{j} \partial_1^{i - j}c(\varphi, \bz) \partial_1^j(f \circ \varphi)(\bz),\,\,\bz \in \D^2.
\end{equation}
By the Fa\`a di Bruno's formula, we have
\begin{equation*}
\partial_1^j(f \circ \varphi)(\bz) = \sum \frac{j!}{l_1!  l_2! \cdots l_j!} \partial_1^{l_1 + \cdots + l_j} f(\varphi(\bz)) \displaystyle \prod_{k=1}^j \left( \frac{\varphi^{(k)}(z_1)}{k!}\right)^{l_k},
\end{equation*}
where the sum is over all $j$-tuples of non-negative integers $(l_1, \ldots, l_j)$ such that $1 \cdot l_1 + 2 \cdot l_2 + \cdots + j\cdot l_j = j$.

Since $f \in \mathcal{M}_m$ and $\varphi(\Delta) = \Delta$, it follow that if $\bz \in \Delta$, then $\left(\partial_1^k f\right)(\varphi(\bz)) = 0$ for $0 \leq k \leq m$. Thus, if $0 \leq j \leq m$ and $\bz \in \Delta$, then we have $\partial_1^{l_1 + \cdots + l_j}f(\varphi(\bz)) = 0$ for all $j$-tuples of non-negative integers $(l_1,\ldots, l_j)$ with $1 \cdot l_1 + 2 \cdot l_2 + \cdots + j\cdot l_j = j$. Therefore, from the Fa\`a di Bruno's formula, we have $\partial_1^j(f \circ \varphi)(\bz) = 0$ for every $\bz \in \Delta$. Now from Equation \eqref{eq:4}, it follows that $\partial_1^i \left(\pi(\varphi^{-1})f\right)|_{\Delta} = 0$. This proves that $\mathcal{M}_m$ is an invariant subspace of the representation $\pi$ for each $k \geq 0$.

(ii) Assume that there exist $z_0 \in \mathbb D$ and $\varphi_0 \in$ M\"ob such that $c(\varphi_0, (z_0, z_0)) = 0$. Let $\psi$ be an arbitrary element in M\"ob. Then, we have 
$$c(\psi, (z_0, z_0)) = c(\psi \circ \varphi_0^{-1} \circ \varphi_0, (z_0, z_0)) = c(\psi \circ \varphi_0^{-1}, (\varphi_0(z_0), \varphi_0(z_0)))c(\varphi_0, (z_0, z_0)) = 0.$$
Thus, for every $\psi$ in M\"ob, $c(\psi, (z_0, z_0)) = 0$. 

Take an arbitrary $\varphi$ in M\"ob and an arbitrary $z \in \mathbb D$. Suppose $\tilde{\psi} \in$ M\"ob be such that $\tilde{\psi}(z) = z_0$. Then, we have
$$c(\varphi, (z,z)) = c(\varphi \circ \tilde{\psi}^{-1} \circ \tilde{\psi}, (z,z)) = c(\varphi \circ \tilde{\psi}^{-1}, (z_0, z_0))c(\tilde{\psi}, (z,z)) = 0.$$
Thus, we have proved that $c(\varphi, (z, z)) = 0$ for every $\varphi \in$ M\"ob and $z \in \mathbb D$.

Now, we prove that $\pi(\varphi) = 0$ for every $\varphi \in$ M\"ob. Due to part (i), it is enough to prove that $\pi(\varphi)|_{\mathcal{S}_m} = 0$ for every $\varphi \in$ M\"ob. Let $f \in \mathcal{S}_m$ be an arbitrary element. For $\varphi \in \mbox{M\"ob}$ and $z \in \mathbb D$, we have
\begin{flalign*}\label{neweq:12}
\nonumber \partial_1^m \left(\pi(\varphi^{-1}) f \right)(z,z)
\nonumber &= \partial_1^m \left(c(\varphi, (z,z)) (f \circ \varphi)(z,z)\right)\\
\nonumber &= \displaystyle \sum_{i = 0}^m \partial_1^{m-i} c(\varphi, (z,z)) \partial_1^i (f \circ \varphi)(z,z)\\
&= c(\varphi, (z,z)) \partial_1^m (f \circ \varphi)(z,z)\\
&= 0.
\end{flalign*}
The third equality follows from the fact that $f \in \mathcal{S}_m$ and therefore, $\partial_1^i f|_{\Delta} = 0$ for $0 \leq i \leq m-1$. As a consequence of the first part, we obtain that $\pi(\varphi^{-1})f = 0$. This is a contradiction, since $\pi$ is assumed to be a non-trivial representation.
\end{proof}

We now realize the Hilbert space $\mathcal{S}_m$ as a Hilbert space consisting of holomorphic functions over $\mathbb D$ and provide a formal expression for the reproducing kernel of $\mathcal{S}_m$. In order to do so, we need a few technical lemmas. Note that the set $\left\lbrace \overbar{\partial_1}^i K(\cdot, \bw) : 0 \leq i \leq m,\,\,\bw \in \Delta\right\rbrace$ is dense in $\mathcal{M}_m^{\perp}$ for each $m \geq 0$. Let $P_m$ denote the orthogonal projection onto $\mathcal{S}_m$. The following lemma gives us a spanning set of $\mathcal{S}_m$.

\begin{lem}\label{lem1}
For each $m \geq 0$, the set $\left\lbrace P_m\overbar{\partial_1}^m K(\cdot, \bw) : \bw \in \Delta\right\rbrace$ is dense in $\mathcal{S}_m$.
\end{lem}

\begin{proof}
First, assume that $m = 0$. Since $K(\cdot, \bw) \in \mathcal{S}_0 = \mathcal{M}_0^\perp$ for every $\bw \in \Delta$, it follows that $P_0K(\cdot, \bw) = K(\cdot, \bw)$ for every $\bw \in \Delta$. Also, it is easy to verify that the set $\left\lbrace K(\cdot, \bw) : \bw \in \Delta\right\rbrace$ is dense in $\mathcal{S}_0$.

Now assume that $m \geq 1$. Let $f \in \mathcal{S}_m$ be such that $\left\langle f, P_m\overbar{\partial_1}^m K(\cdot, \bw) \right\rangle = 0$ for every $\bw \in \Delta$. Since $\left\langle f, P_m\overbar{\partial_1}^m K(\cdot, \bw) \right\rangle = \left\langle f, \overbar{\partial_1}^m K(\cdot, \bw) \right\rangle = \partial_1^mf(\bw)$, it follows that $\partial_1^mf(\bw) = 0$ for every $\bw \in \Delta$.

Recall that $f \in \mathcal{S}_m \subseteq \mathcal{M}_{m-1}$ and thus, we have $\partial_1^i f|_{\Delta} = 0$ for $0 \leq i \leq m-1$. From the above discussion, it follows that $\partial_1^i f|_{\Delta} = 0$ for $0 \leq i \leq m$ and therefore $f \in \mathcal{M}_m$. This implies that $f = 0$. This proves that the set $\left\lbrace P_m\overbar{\partial_1}^m K(\cdot, \bw) : \bw \in \Delta\right\rbrace$ is dense in $\mathcal{S}_m$.
\end{proof}

Note that there is no guarantee for a subspace $\mathcal{S}_m$ being non-zero. However, the following lemma describes a criteria for the non-emptiness of $\mathcal{S}_m$.

\begin{lem}\label{lem4}
For $m \geq 0$, the subspace $\mathcal{S}_m$ is non-zero if and only if $P_m \overbar{\partial_1}^m K(\cdot, (0,0)) \neq 0$.
\end{lem}

\begin{proof}
Converse direction of the statement is straightforward. So, we prove the forward direction only. Assume that $P_m \overbar{\partial_1}^m K(\cdot, (0,0)) = 0.$ Let $f \in \mathcal{S}_m$. Then, we have $\partial_1^j f|_{\Delta} = 0$, $0 \leq j \leq m-1$. For $\varphi \in$ M\"ob, we have 
\begin{flalign*}
0=\left\langle \pi(\varphi^{-1})f, P_m \overbar{\partial_1}^m K(\cdot, (0,0))\right\rangle &= \left\langle \pi(\varphi^{-1})f, \overbar{\partial_1}^m K(\cdot, (0,0))\right\rangle\\
&= \left\langle c(\varphi, \cdot) f \circ \varphi, \overbar{\partial_1}^m K(\cdot, (0,0))\right\rangle\\
&= \partial_1^m \left(c(\varphi, (0,0)) (f \circ \varphi)(0,0)\right)\\
&= \displaystyle \sum_{i = 0}^m \partial_1^{m-i} c(\varphi, (0,0)) \partial_1^i (f \circ \varphi)(0,0).
\end{flalign*}
Since $(\varphi(0), \varphi(0))$ is in the diagonal set $\Delta$, an application of the Fa\`a di Bruno's formula implies that $\partial_1^i (f \circ \varphi)(0,0) = 0$ for every $0 \leq i\leq m-1$.

Now, expanding the term $\partial_1^m (f \circ \varphi)(0,0)$ using Fa\`a di Bruno's formula, we have 
\begin{equation}\label{eq:22}
\partial_1^m(f \circ \varphi)(0,0) = \sum \frac{m!}{l_1!  l_2! \cdots l_m!} \partial_1^{l_1 + \cdots + l_m} f(\varphi(0),\varphi(0)) \displaystyle \prod_{k=1}^m \left( \frac{\varphi^{(k)}(0)}{k!}\right)^{l_k},
\end{equation}
where the sum is over all $m$-tuples of non-negative integers $(l_1, \ldots, l_m)$ such that $1 \cdot l_1 + 2 \cdot l_2 + \cdots + m\cdot l_m = m$.

Since $f \in \mathcal{S}_m$, we have $\partial_1^{l_1 + \cdots + l_m} f(\varphi(0),\varphi(0)) = 0$ if $l_1 + \cdots + l_m < m$. Furthermore, if $(l_1, \ldots, l_m)$ is a tuple of non-negative integers such that $l_1 + \cdots + l_m = m$ and $1 \cdot l_1 + 2 \cdot l_2 + \cdots + m\cdot l_m = m$, then $l_1 = m$ and $l_i = 0$ for $2 \leq i \leq m$. Thus, form Equation \eqref{eq:22}, we have $\partial_1^m(f \circ \varphi)(0,0) = (\varphi'(0))^m \partial_1^mf(\varphi(0),\varphi(0))$ and therefore 
$$0=\displaystyle \sum_{i = 0}^m \partial_1^{m-i} c(\varphi, (0,0)) \partial_1^i (f \circ \varphi)(0,0) = c(\varphi, (0,0))(\varphi'(0))^m \left(\partial_1^mf\right)(\varphi(0),\varphi(0)).$$
By the part (ii) of Proposition \ref{prop1}, we have $c(\varphi, (0,0)) \neq 0$. Thus, from the previous equation, it follows that $\left(\partial_1^m f\right)(\varphi(0),\varphi(0)) = 0$. Note that the action of M\"ob on the diagonal set $\Delta$ is transitive. Thus, we have $\partial_1^m f|_{\Delta} = 0$ and therefore $f \in \mathcal{M}_m$. This proves that $f = 0$.
\end{proof}

\begin{rem}\label{rem1}
(i) A possible application of the description of all the multiplier representations of M\"ob on reproducing kernel Hilbert spaces consisting of holomorphic functions over $\D^2$ is to describe all the pairs of M\"ob - homogeneous operators in $B_1(\D^2)$. Therefore, it is natural to assume that the set of all polynomials over $\D^2$ lies in the Hilbert space. 

(ii) If $\mathcal{H}$ contains the set of all polynomials over $\D^2$, then it is easy to see that $\mathcal{S}_m \neq \{0\}$ for every $m \geq 0.$ 
\end{rem}

For $m \geq 0$, let $\Gamma_m : \mathcal{S}_m \to \mbox{Hol}(\D, \C)$ be the map defined by 
\begin{equation}\label{eq:5}
\Gamma_m f = \partial_1^m f|_{\Delta},\,\,f \in \mathcal{S}_m.
\end{equation}
Evidently, the map $\Gamma_m$ is an injective linear map. Let $\mathcal{A}_m := \mbox{Im}\,\, \Gamma_m$. Define an inner product on $\mathcal{A}_m$ such that $\Gamma_m : \mathcal{S}_m \to \mathcal{A}_m$ is unitary. 
\begin{prop}\label{prop2}
For each $m \geq 0$, the Hilbert space $\mathcal{A}_m$ is a reproducing kernel Hilbert space with the reproducing kernel $K_m : \D \times \D \to \C$, defined by,
\begin{equation}\label{eq:18}
K_m(z, w) = \left\langle \Gamma_m P_m \overbar{\partial_1}^m K(\cdot, (w,w)), \Gamma_m P_m \overbar{\partial_1}^m K(\cdot, (z,z))\right\rangle,\,\,z, w \in \D.
\end{equation}
\end{prop}

\begin{proof}
For $w \in \D$, let $K_{m}(\cdot, w) = \Gamma_m P_m \overbar{\partial_1}^m K(\cdot, (w,w)).$ Suppose $g \in \mathcal{A}_m$. Then there exists $f \in \mathcal{S}_m$ such that $g = \Gamma_m f = \partial_1^m f|_{\Delta}$. For $w \in \D$, we have 
\begin{flalign*}
\left\langle g, K_m(\cdot, w)\right\rangle &= \left\langle \Gamma_m f, \Gamma_m P_m \overbar{\partial_1}^m K(\cdot, (w,w))\right\rangle\\
&= \left\langle f, P_m \overbar{\partial_1}^m K(\cdot, (w,w))\right\rangle\\
&= \left\langle f, \overbar{\partial_1}^m K(\cdot, (w,w))\right\rangle\\
&= \partial_1^m f(w,w)\\
&= g(w).
\end{flalign*}
Here, the second equality holds because the map $\Gamma_m$ is unitary and the third equality holds because $f \in \mathcal{S}_m$.

Also, from Lemma \ref{lem1}, it follows that the set $\left\lbrace K_m(\cdot, w) : w\in \D\right\rbrace$ is dense in $\mathcal{A}_m$. Therefore, $\mathcal{A}_m$ is a reproducing kernel Hilbert space with the reproducing kernel $K_m : \D \times \D \to \C$, defined by Equation \eqref{eq:18}.
\end{proof}

The map $\pi_m : \mbox{M\"ob} \to \mathcal{U}(\mathcal{A}_m)$, defined by,
\begin{equation}\label{eqn: xx}
\pi_m(\varphi) := \Gamma_m \left(\pi(\varphi)|_{\mathcal{S}_m}\right) \Gamma_m^*,\,\,\varphi \in \mbox{M\"ob}.    
\end{equation}
is a unitary representation of M\"ob on the reproducing kernel Hilbert space $\mathcal{A}_m$. We prove that each $\pi_m$ is an irreducible representation. To prove that each $\pi_m$ is an irreducible representation, we need the following lemma which is also given in \cite{KobIrr}. However, for the completeness of the paper, we include a proof of the following lemma.
\begin{lem}\label{lem2}
Let $G$ be group and $\Omega \subseteq \C^n$ be a domain such that $G$ acts on $\Omega$ transitively. Suppose $\rho : G \to \mathcal{U}(\mathcal{B})$ is a multiplier representation on a reproducing kernel Hilbert space $\mathcal{B}$ consisting of scalar valued holomorphic functions on $\Omega$. Then the representation $\rho$ is irreducible. 
\end{lem}

\begin{proof}
Let $L : \Omega \times \Omega \to \C$ be the reproducing kernel of $\mathcal{B}$ and $c_\rho : G \times \Omega \to \C$ be the corresponding cocycle of $\rho$. Since $\mathcal{B}$ consists of holomorphic functions, it follows that $L$ is holomorphic in first variable and anti-holomorphic in second variable. 

For $f \in \mathcal{B}$, we have $(\rho(g^{-1})f)(\bz) = c_\rho(g, \bz)f(g\cdot\bz),$ $g \in G$, $\bz \in \Omega$. A straightforward calculation shows that 
\begin{equation}\label{eq:6}
\rho(g^{-1})^* L(\cdot, \bw) = \overbar{c_\rho(g, \bw)} L(\cdot, g\cdot \bw),
\end{equation}
for every $g \in G$ and $\bw \in \Omega$.
Note that $\rho(g^{-1})$ is a unitary operator. Therefore, it follows from Equation \eqref{eq:6} that the reproducing kernel $L$ satisfies the following transformation rule
\begin{equation}\label{eq:7}
L(\bz, \bw) = c_\rho(g, \bz) L(g\cdot \bz, g \cdot \bw) \overbar{c_\rho(g, \bw)},\,\,g \in G,\,\,\bz, \bw \in \Omega. 
\end{equation}
Without loss of generality, assume that $0 \in \Omega$. Since the action of $G$ is transitive on $\Omega$, for every $\bz \in \Omega$, there exists $g_{\bz} \in G$ such that $g_{\bz} \cdot \bz = 0$. Substituting $\bz = \bw$ and then replacing $g$ by $g_{\bz}$ in Equation \eqref{eq:7}, we obtain
\begin{equation}\label{eq:8}
L(\bz, \bz) = c_\rho(g_{\bz}, \bz) L(0,0) \overbar{c_\rho(g_{\bz}, \bz)},\,\,\bz \in \Omega. 
\end{equation}
Let $\mathcal{M}$ be a closed reducing subspace of the representation $\rho$. Since $\mathcal{M}$ is a closed subspace of a reproducing kernel Hilbert space, it follows that $\mathcal{M}$ is also a reproducing kernel Hilbert space. Let $L_{\mathcal{M}}$ be the reproducing kernel of $\mathcal{M}$. Since $\rho : G \to \mathcal{U}(\mathcal{M})$ is a unitary representation, the reproducing kernel $L_{\mathcal{M}}$ also satisfies the following equation
\begin{equation}\label{eq:9}
L_{\mathcal{M}}(\bz, \bz) = c_\rho(g_{\bz}, \bz) L_{\mathcal{M}}(0,0) \overbar{c_\rho(g_{\bz}, \bz)},\,\,\bz \in \Omega. 
\end{equation} 
From Equations \eqref{eq:8} and \eqref{eq:9}, it follows that if $\alpha = \frac{L_{\mathcal{M}}(0,0)}{L(0,0)}$, then 
\begin{equation*}\label{eq:10}
L_{\mathcal{M}}(\bz, \bz) = \alpha L(\bz, \bz),\,\,\bz \in \Omega. 
\end{equation*}
Polarizing the above equation, we have $L_{\mathcal{M}}(\bz, \bw) = \alpha L(\bz, \bw),\,\,\bz, \bw \in \Omega$. If $\alpha = 0$, then $\mathcal{M} = \{0\}$ and if $\alpha > 0$, then $\mathcal{M} = \mathcal{B}$. This proves that $\rho$ is irreducible.
\end{proof}

Employing the aforementioned lemma, we first prove that each of the representations $\pi_m$, defined in Equation \eqref{eqn: xx}, is an irreducible multiplier representation. We also derive the reproducing kernel for the Hilbert space $\mathcal{A}_m$.


\begin{prop}\label{thm1}
For every $m \geq 0$, the representation $\pi_m$ is a multiplier representation on the reproducing kernel Hilbert space $\mathcal{A}_m$ with the associated cocycle $c_m : \mbox{M\"ob} \times \D \to \C$, defined by $c_m(\varphi, z) = c(\varphi, (z,z))\left(\varphi'(z)\right)^m,\,\,\varphi \in \mbox{M\"ob},\,\,z \in \D$. Each of the representations $\pi_m$ is irreducible. The reproducing kernel $K_m$ of the Hilbert space $\mathcal{A}_m$ is obtained by polarizing the following expression
\begin{equation*}\label{eq:11}
c_m(\varphi_z, z)K_m(0,0) \overbar{c_m(\varphi_z, z)},
\end{equation*} 
where $\varphi_z \in$ M\"ob, defined by $\varphi_z(w) = \frac{z-w}{1 - \bar z w},\,\,w \in \mathbb D$, is the unique involution mapping $z$ to $0$.
\end{prop}

\begin{proof}
Let $g \in \mathcal{A}_m$. Then there exists $f \in \mathcal{S}_m$ such that $g = \Gamma_m f$. For $\varphi \in \mbox{M\"ob}$, we have
\begin{flalign}\label{eq:12}
\nonumber \pi_m(\varphi^{-1}) g &= \Gamma_m \pi(\varphi^{-1}) \Gamma_m^* \Gamma_m f\\
\nonumber &= \Gamma_m \pi(\varphi^{-1})f\\
\nonumber &= \Gamma_m c(\varphi, \cdot) (f \circ \varphi)\\
\nonumber &= \partial_1^m \left(c(\varphi, \cdot) (f \circ \varphi)\right)|_{\Delta}\\
&= \displaystyle \sum_{i = 0}^m \partial_1^{m-i} c(\varphi, \cdot) \partial_1^i (f \circ \varphi)|_{\Delta}
\end{flalign}
Since $f \in \mathcal{S}_m$, we have $\partial_1^j f|_{\Delta} = 0$ for $0 \leq j \leq m-1$. 
Now, following a similar argument as given in the proof of Lemma \ref{lem4}, it follows that the only non-zero term in $\partial_1^m (f \circ \varphi)(z,z)$ is $\left(\left(\partial_1^m f\right) (\varphi(z), \varphi(z))\right) (\varphi'(z))^m$ for $z \in \mathbb D$. Thus, from Equation \eqref{eq:12}, we have 
\begin{equation*}\label{eq:13}
\pi_m (\varphi^{-1}) g(z) = c(\varphi, (z,z))(\varphi'(z))^m g(\varphi(z)),\,\,z \in \D.
\end{equation*}
This proves that $\pi_m$ is a multiplier representation of M\"ob on $\mathcal{A}_m$ with the associated cocycle $c_m : \mbox{M\"ob} \times \D \to \C$, defined by, $c_m(\varphi, z) = c(\varphi, (z,z))\left(\varphi'(z)\right)^m,\,\,\varphi \in \mbox{M\"ob},\,\,z \in \D$.

Since the action of the group M\"ob on the unit disc $\D$ is transitive, the irreducibility of the representation $\pi_m$ follows from Lemma \ref{lem2}.

From the proof of Lemma \ref{lem2}, it follows that the reproducing kernel $K_m$ satisfies the transformation rule \eqref{eq:7}. In particular, the equation
\begin{equation}\label{eq:15}
K_m(z,z) = c_m(\varphi_z, z)K_m(0,0) \overbar{c_m(\varphi_z, z)},\,\,z \in \D
\end{equation}
holds, where $\varphi_z$ is the unique involution of M\"ob which maps $z$ to $0$.
Therefore, the reproducing kernel $K_m$ of the Hilbert space $\mathcal{A}_m$ is obtained by polarizing the expression \eqref{eq:15}.
\end{proof}

The following theorem describes the representation $\pi$ completely by identifying each of the multiplier representations $\pi_m$ of M\"ob into the reproducing kernel Hilbert space $\mathcal{A}_m$.

\begin{thm}\label{thm3}
Let $\pi : \mbox{M\"ob} \to \mathcal{U}(\mathcal{H})$ be a multiplier representation where $\mathcal{H}$ is a reproducing kernel Hilbert space.

(i) If $\mathcal{H}$ consists of holomorphic functions on $\D^2$ and contains all the polynomials, then there exists $\lambda > 0$ such that $\pi \cong \oplus_{m = 0}^{\infty} D^+_{\l + 2m}$.

(ii) If $\mathcal{H}$ consists of symmetric (resp. anti-symmetric) holomorphic functions over $\D^2$ and contains all the symmetric (resp. anti-symmetric) polynomials, then there exists $\lambda > 0$ such that $\pi \cong \oplus_{m = 0}^\infty D^+_{\lambda + 4m}$ (resp. $\pi \cong \oplus_{m=0}^\infty D^+_{\lambda + 4m + 2}$).
\end{thm}

\begin{proof}
(i) For every $m \geq 0$, let $\mathcal{M}_m$ be the subspace defined in Equation \eqref{eq:14} and $\mathcal{S}_m$ be the subspace defined in the line below the equation \eqref{eq:14}. Also, suppose $\Gamma_m$ is the map defined in Equation \eqref{eq:5} and $\mathcal{A}_m :=\mbox{Im}\,\,\Gamma_m$.

From Proposition \ref{prop2}, it follows that $\mathcal{A}_m$ is a reproducing kernel Hilbert space. Let $K_m$ be the reproducing kernel of $\mathcal{A}_m$. Now, Proposition \ref{thm1} and Lemma \ref{lem2} together yield that the representation $\pi_m : \mbox{M\"ob} \to \mathcal{U}(\mathcal{A}_m)$, defined by $\pi_m(\varphi) = \Gamma_m \left(\pi(\varphi)|_{\mathcal{S}_m}\right) \Gamma_m^*,$ $\varphi \in$ M\"ob, is an irreducible multiplier representation with the associated cocycle $c_m : \mbox{M\"ob} \times \D \to \C$, defined by, $c_m(\varphi, z) = c(\varphi, (z,z))\left(\varphi'(z)\right)^m,\,\,\varphi \in \mbox{M\"ob},\,\,z \in \D$. Note that $c_0 (\varphi, z) = c(\varphi, (z,z))$, $\varphi \in$ M\"ob, $z \in \D$.  

Since $\pi_0 : \mbox{M\"ob} \to \mathcal{U}(\mathcal{A}_0)$ is a multiplier representation and $\mathcal{A}_0$ is a reproducing kernel Hilbert space, the reproducing kernel $K_0$ of $\mathcal{A}_0$ satisfy Equation \eqref{eq:7}. Now, an application of \cite[Proposition 2.1]{HHHVB1} yields that 
\begin{equation}\label{neweq3}
 \partial \bar{\partial} \log K_0(z,z) = \frac{\lambda}{(1 - |z^2|)^2},\,\,z \in \mathbb D,   
\end{equation}
for some $\lambda > 0$. On the other hand, a straightforward computation shows that 
\begin{equation}\label{neweq4}
 \partial \bar{\partial} \log B^{(\lambda)}(z,z) = \frac{\lambda}{(1 - |z^2|)^2},\,\,z \in \mathbb D,   
\end{equation}
where $B^{(\lambda)} (z, w) = \frac{1}{(1 - z \bar w)^\lambda}$, $z, w \in \mathbb D$, is the reproducing kernel of $\mathbb A^{(\lambda)}(\mathbb D)$. Equations \eqref{neweq3} and \eqref{neweq4} together imply the existence of a non-vanishing holomorphic function $F : \mathbb D \to \mathbb C$ such that 
\begin{equation}\label{neweq5}
   K_0(z, w) = F(z) B^{(\lambda)}(z, w) \overline{F(w)},\,\,z, w \in \mathbb D. 
\end{equation}
For any $m \geq 1$, substituting the expression of the cocyle $c_m$ in Equation \eqref{eq:15}, we obtain 
\begin{equation}\label{eq:21}
K_{m}(z,w) = \frac{K_{m}(0,0)}{K_0(0,0)} F(z) B^{(\lambda + 2m)}(z, w)\overline{F(w)},\,\,z, w \in \mathbb{D}.
\end{equation}

Note that $\mathcal{H}$ contains all the polynomials over $\mathbb D^2$, and therefore, it follows form Remark \ref{rem1}(ii) and Lemma \ref{lem4} that $K_m(0,0) > 0$ for every $m \geq 0$. The equation \eqref{eq:21} yields that the map $U_m : \mathcal{A}_m \to \mathbb A^{(\lambda + 2m)}(\mathbb D)$, defined by 
$$U_m K_m(\cdot, w) = \left(\frac{K_m(0,0)}{K_0(0,0)}\right)^\frac{1}{2} B^{(\lambda + 2m)}(\cdot, w) \overline{F(w)},\,\,w \in \mathbb D$$
is a unitary operator. A direct computation verifies that the unitary operator $U_m$ intertwines the representations $\pi_m$ and the holomorphic discrete series representation $D_{\l+2m}^+$. As a result, the representation $\pi_m$ is equivalent to $D_{\l+2m}^+$ and consequently, $\pi \cong \oplus_{m = 0}^{\infty} D^+_{\l + 2m}.$

(ii) Suppose $\mathcal{H}$ is a Hilbert space consisting of symmetric holomorphic functions. We claim that $\mathcal{S}_k = \{0\}$ if $k$ is odd. 

For an odd natural number $k$, let $f \in \mathcal{S}_k$ be a non-zero element. Then $f \in \mathcal{M}_{k-1}$, but $f \notin \mathcal{M}_k$. This implies that $\partial_1^i f|_{\Delta} = 0$ for each $0 \leq i \leq k-1$ and $\partial_1^k f|_{\Delta} \neq 0$. It follows that there exists a holomorphic function $g : \D^2 \to \C$ such that 
$$f(z_1, z_2) = (z_1 - z_2)^k g(z_1, z_2),\,\,z_1, z_2 \in \D.$$
Since $k$ is an odd natural number and $f$ is a symmetric function on $\D^2$, the function $g$ must be anti-symmetric and therefore $g|_{\Delta} = 0$. This implies that $\partial_1^k f|_{\Delta} = 0$, which is a contradiction. Thus, we must have $\mathcal{S}_k = \{0\}$ if $k$ is an odd natural number.

We have already proved that there exists $\lambda > 0$ such that $\pi_k \cong D_{\lambda + 2k}^+$ for every even number $k$.
Therefore, $\pi \cong \displaystyle \oplus_{m = 0}^\infty D_{\lambda + 4m}^+$.

Similarly, if $\mathcal{H}$ is a reproducing kernel Hilbert space consisting of anti-symmetric holomorphic functions over $\D^2$ and contains all the anti-symmetric polynomials over $\D^2$, then there exists a $\lambda > 0$ such that $\pi \cong \displaystyle \oplus_{m = 0}^\infty D_{\lambda + 4m + 2}^+$.  
\end{proof}

For any $\l_1, \l_2 > 0$, $D_{\l_1}^+ \otimes D_{\l_2}^+$ defines an irreducible unitary representation of M\"ob $\times$ M\"ob on the Hilbert space $\mathbb{A}^{(\l_1)}(\D) \otimes \mathbb{A}^{(\l_2)}(\D)$. Although, $D_{\l_1}^+ \otimes D_{\l_2}^+$ is an irreducible representation of M\"ob $\times$ M\"ob, the restriction $D_{\l_1}^+ \otimes D_{\l_2}^+|_{\mbox{M\"ob}}$ is not an irreducible representation, in fact 
\begin{equation}\label{eq:1}
D_{\l_1}^+ \otimes D_{\l_2}^+|_{\mbox{M\"ob}} \cong \oplus_{n = 0}^{\infty} D^+_{\l_1 + \l_2 + 2n}.
\end{equation}
This is known as Clebsh-Gordan formula. An immediate corollary of the Theorem \ref{thm3} gives us the proof of the Clebsh-Gordan formula.
\begin{cor}\label{cor1}
(i) For $\l_1, \l_2 > 0$, let $D_{\l_i}^+$ be holomorphic discrete series representation of M\"ob on the weighted Bergman space $\mathbb{A}^{(\lambda_i)}(\D)$. Then, we have $D_{\l_1}^+ \otimes D_{\l_2}^+|_{\mbox{M\"ob}} \cong \oplus_{n = 0}^{\infty} D^+_{\l_1 + \l_2 + 2n}.$

(ii) For any $\l > 0$, let $\mathbb A_{\mbox{\tiny{sim}}}^{(\l)}(\D^2) = \{f \in \mathbb{A}^{\l}(\D) \otimes \mathbb{A}^{\l}(\D) : f(z_1, z_2) = f(z_2, z_1),\,\,z_1, z_2 \in \D\}$ and $\mathbb A_{\mbox{\tiny{anti}}}^{(\l)}(\D^2) = \{f \in \mathbb{A}^{\l}(\D) \otimes \mathbb{A}^{\l}(\D) : f(z_1, z_2) = -f(z_2, z_1),\,\,z_1, z_2 \in \D\}$. Then, $D_\l^+ \otimes D_\l^+ : \mbox{M\"ob} \to \mathcal{U}\left(\mathbb  A_{\mbox{\tiny{sim}}}^{(\l)}(\D^2)\right)$ (resp. $D_\l^+ \otimes D_\l^+ : \mbox{M\"ob} \to \mathcal{U}\left(\mathbb  A_{\mbox{\tiny{anti}}}^{(\l)}(\D^2)\right)$) is unitarily equivalent to $\oplus_{m = 0}^{\infty} D^+_{2\l + 4m}$ (resp. $\oplus_{m = 0}^{\infty} D^+_{2\l + 4m + 2}$).
\end{cor}

Combining the first part of the corollary above and Theorem \ref{thm3}, we obtain the following corollary.

\begin{cor}\label{thm3x}
Let $\pi : \mbox{M\"ob} \to \mathcal{U}(\mathcal{H})$ be a multiplier representation where $\mathcal{H}$ is a reproducing kernel Hilbert space. If $\mathcal{H}$ consists of holomorphic functions on $\D^2$ and contains all the polynomials, then there exists $\lambda_1, \lambda_2 > 0$ such that $\pi \cong D^+_{\lambda_1} \otimes D^+_{\lambda_2}|_{\mbox{M\"ob}}$. 
\end{cor}

\subsection{The case of arbitrary $n$}

We now describe the representation $\pi$ when $n$ is arbitrary. To accomplish this, let us consider the subset 
$$\Delta_{n} := \{(z_1,z_1,z_2,z_3,\hdots ,z_{n-1}) : z_1,  \ldots, z_{n-1} \in \D\}$$
of $\D^n$.
For every $k_n \geq 0$, suppose
\begin{equation}\label{eq:32}
\mathcal{M}_{k_n} := \left\lbrace f \in \mathcal{H} : \partial_{1}^i f|_{\Delta_{n}} = 0,\,\,0 \leq i \leq k_n \right\rbrace.
\end{equation}
Each $\mathcal{M}_{k_n}$ is a closed subspace and $\mathcal{M}_{k'_{n}} \subseteq \mathcal{M}_{k_n}$ if ${k'_{n}} \geq {k_n}$. Consider the subspaces $\mathcal{S}_0 = \mathcal{H} \ominus \mathcal{M}_0$ and $\mathcal{S}_{k_n} = \mathcal{M}_{k_{n}-1} \ominus \mathcal{M}_{k_n}$ for $k_{n} \geq 1$. Evidently, we have $\mathcal{H} = \oplus_{k_n \geq 0} \mathcal{S}_{k_n}$.

\begin{prop}\label{prop31}
For each $k_{n} \geq 0$, the subspace $\mathcal{S}_{k_n}$ is a reducing subspace of the representation $\pi$.
\end{prop}

\begin{proof}
 The proof is similar to the proof of Proposition \ref{prop1}.   
\end{proof}

The set $\left\lbrace \overbar{\partial_1}^i K(\cdot, \bw) : 0 \leq i \leq k_{n},\,\,\bw \in \Delta_{n}\right\rbrace$ is dense in $\mathcal{M}_{k_n}^{\perp}$ for each $k_{n} \geq 0$. Let $P_{k_n}$ denote the orthogonal projection onto $\mathcal{S}_{k_n}$. It follows from a similar argument as it is given in the proof of the Lemma \ref{lem1} that the set $\left\lbrace P_{k_n}\overbar{\partial_1}^{k_n} K(\cdot, \bw) : \bw \in \Delta_{n}\right\rbrace$ is dense in $\mathcal{S}_{k_n}$ for every $k_{n} \geq 0$. As before, consider the map $\Gamma_{k_n} : \mathcal{S}_{k_n} \to \mbox{Hol}(\D^{n-1}, \C)$, defined by, 
\begin{equation}\label{eq:34}
\Gamma_{k_n} f = \partial_1^{k_n}f|_{\Delta_n},\,\,f \in \mathcal{S}_{k_n}.
\end{equation}
Evidently, the map $\Gamma_{k_n}$ is an injective linear map. Let $\mathcal{A}_{k_n} := \mbox{Im}\,\, \Gamma_{k_n}$. Define an inner product on $\mathcal{A}_{k_n}$ such that $\Gamma_{k_n} : \mathcal{S}_{k_n} \to \mathcal{A}_{k_n}$ is unitary. The proof of the following proposition is similar to that of Proposition \ref{prop2}.
\begin{prop}\label{prop34}
For each $k_n \geq 0$, the Hilbert space $\mathcal{A}_{k_n}$ is a reproducing kernel Hilbert space with the reproducing kernel $K_{k_n} : \D^{n-1} \times \D^{n-1} \to \C$, defined by,
\begin{align*}
&K_{k_n}((z_1,z_2,\hdots,z_{n-1}), (w_1,w_2,\hdots,w_{n-1}))\\
&\phantom{gada} = \langle \Gamma_{k_n}P_{k_n} \overbar{\partial_1}^{k_n}K(\cdot, (w_1,w_1,w_2,\hdots,w_{n-1})), \Gamma_{k_n}P_{k_n} \overbar{\partial_1}^{k_n}K(\cdot, (z_1,z_1,z_2,\hdots,z_{n-1}))\rangle,
\end{align*}
for $z_{i}, w_{i} \in \D$, $1 \leq i \leq n-1$.
\end{prop}

Let $\pi_{k_n} : \mbox{M\"ob} \to \mathcal{U}(\mathcal{A}_{k_n})$ be the representation, defined by, 
$$\pi_{k_n}(\varphi) := \Gamma_{k_n} \left(\pi(\varphi)|_{\mathcal{S}_{k_n}}\right) \Gamma_{k_n}^*,\,\,\varphi \in \mbox{M\"ob}.$$
Note that each $\pi_{k_n}$ is a unitary representation of M\"ob on the reproducing kernel Hilbert space $\mathcal{A}_{k_n}$. 

\begin{prop}\label{prop35}
For each $k_{n} \geq 0$, the representation $\pi_{k_n}$ is a multiplier representation on the reproducing kernel Hilbert space $\mathcal{A}_{k_n}$ with the associated cocycle $c_{k_n} : \mbox{M\"ob} \times \D^{n-1} \to \C$, defined by, $c_{k_n}(\varphi,(z_1,z_2,\hdots,z_{n-1}) ) = c(\varphi, (z_1,z_1,z_2,\hdots,z_{n-1}))\left(\varphi'(z_1)\right)^{k_n},$ for $\varphi \in \mbox{M\"ob}$ and $(z_{1}, \ldots, z_{n-1}) \in \D^{n-1}$.
\end{prop}

\begin{proof}
The proof is similar to the proof of Proposition \ref{thm1}.  
\end{proof}

Using the mathematical induction method, we now prove the main theorem of this section.

\begin{thm}\label{main thm}
Let $\pi : \mbox{M\"ob} \to \mathcal{U}(\mathcal{H})$ be a multiplier representation where $\mathcal{H}$ is a reproducing kernel Hilbert space. If $\mathcal{H}$ consists of holomorphic functions on $\D^n$ and contains all the polynomials, then there exists $\l_1, \l_2, \hdots, \l_n > 0$ such that  $\pi \cong   (\otimes_{i=1}^{n} D_{\l_i}^+)|_{\mbox{M\"ob}}$.   
\end{thm}

\begin{proof}
For every $k_n \geq 0$, let $\mathcal{M}_{k_n}$ be the subspace defined in Equation \eqref{eq:32} and $\mathcal{S}_{k_n}$ be the subspace defined in the line below Equation \eqref{eq:32}. Also, suppose $\Gamma_{k_n}$ is the map defined in Equation \eqref{eq:34} and $\mathcal{A}_{k_n} : =\mbox{Im}\,\,\Gamma_{k_n}$.

From Proposition \ref{prop34}, it follows that $\mathcal{A}_{k_n}$ is a reproducing kernel Hilbert space. Also, Proposition \ref{prop35} yields that the representation $\pi_{k_n} : \mbox{M\"ob} \to \mathcal{U}(\mathcal{A}_{k_n})$, defined by $\pi_{k_n}(\varphi) = \Gamma_{k_n} \left(\pi(\varphi)|_{\mathcal{S}_{k_n}}\right) \Gamma_{k_n}^*,$ $\varphi \in$ M\"ob, is a multiplier representation with the associated cocycle $c_{k_n} : \mbox{M\"ob} \times \D^{n-1} \to \C$, defined by, 
$$c_{k_n}(\varphi,(z_1,z_2,\hdots,z_{n-1}) ) = c(\varphi, (z_1,z_1,z_2,\hdots,z_{n-1}))\left(\varphi'(z_1)\right)^{k_n},$$
for $\varphi \in \mbox{M\"ob}$ and $(z_{1}, \ldots, z_{n-1}) \in \D^{n-1}$.

The map $\Gamma_{k_n}$ has a natural extension $\tilde{\Gamma}_{k_n} : \mathcal{H} \to \mbox{Hol}(\D^{n-1}, \C)$, given by the formula
$$\tilde{\Gamma}_{k_n} f = \partial_1^{k_n}f|_{\Delta_n},\,\,f \in \mathcal{H}.$$
Due to the assumption $\mathcal{H}$ contains all the polynomials, and therefore, $\mbox{Im}\,\,\tilde{\Gamma}_{k_n}$ also contains all the polynomials on $\mathbb D^{n-1}$. Since $\mathcal{A}_{k_n} = \mbox{Im}\,\,\tilde{\Gamma}_{k_n}$, it follows that the Hilbert space $\mathcal{A}_{k_n}$ contains all the polynomials on $\mathbb D^{n-1}$.

Thus, the representation $\pi_{k_n} : \mbox{M\"ob} \to \mathcal{U}(\mathcal{A}_{k_n})$ is a multiplier representation on a reproducing kernel Hilbert space $\mathcal{A}_{k_n}$ consisting of holomorphic function of $\mathbb D^{n-1}$. Also, $\mathcal{A}_{k_n}$ contains all the polynomials on $\mathbb D^{n-1}$. To apply the mathematical induction, we assume the followings
\begin{itemize}
\item[(i)] $\mathcal{A}_{k_n} = \displaystyle \oplus_{\substack{k_i \geq 0,\\ 2\leq i\leq n-1}} \mathcal{S}_{k_n, k_{n-1}, \ldots, k_2}$,
\item[(ii)] Each $\mathcal{S}_{k_n, k_{n-1}, \ldots, k_2}$ is a reducing subspace of $\pi_{k_n}$,
\item[(iii)] ${\pi_{k_n}}|_{\mathcal{S}_{k_n, \ldots, k_2}} \cong \pi_{k_{n}, \ldots, k_2}$ where $\pi_{k_n, \ldots, k_2}$ is a multiplier representation on a reproducing kernel Hilbert space $\mathcal{A}_{k_n, \ldots, k_2} \subseteq \mbox{Hol}(\mathbb D, \mathbb C)$ with the associated cocycle $c_{k_n, \ldots, k_2} : \mbox{M\"ob} \times \mathbb D \to \mathbb C$, defined by
$$c_{k_n, \ldots, k_2}(\varphi, z) = c(\varphi, (z, \ldots, z)) \left(\varphi'(z)\right)^{k_n + \cdots + k_2},$$
$\varphi \in$ M\"ob, $z \in \D$.
\end{itemize}
Now, following a similar argument as it is given in the proof of the Theorem \ref{thm3}(i), we obtain a $\lambda > 0$ such that
$$\pi_{k_{n}, \ldots, k_2} \cong D^+_{\lambda + 2k_n + \cdots + 2k_{2}},$$
for every $k_i \geq 0$, $2 \leq i \leq n$ and therefore,
$$\pi \cong \oplus_{\substack{k_i \geq 0,\\ 2\leq i\leq n}}  D^+_{\lambda + 2k_n + \cdots + 2k_{2}}.$$
On the other hand, by repeated application of the Clebsh-Gordan formula, we have  
$$\oplus_{\substack{k_i \geq 0,\\ 2\leq i\leq n}}  D^+_{\lambda + 2k_n + \cdots + 2k_{2}} \cong D^+_{\l_1} \otimes \cdots \otimes D^+_{\l_n}|_{\mbox{\tiny{M\"ob}}},$$
for any $\lambda_1, \ldots, \l_n > 0$ such that $\sum_{1 \leq i \leq n} \l_i = \l$. This proves that there exists $\lambda_1, \ldots, \l_n > 0$ such that $\pi \cong D^+_{\l_1} \otimes \cdots \otimes D^+_{\l_n}|_{\mbox{\tiny{M\"ob}}}.$
\end{proof}

\section{Application to homogeneous pairs}\label{section 3}
In a recent paper \cite{HHHVB1}, a classification of all the irreducible tuples of operators in the Cowen-Douglas class $B_r(\D^n)$ for $r = 1, 2, 3$ which are homogeneous with respect to M\"ob$^n$ (direct product of $n$ - copies of M\"ob) as well as Aut$(\D^n)$ is obtained. Also, a family of irreducible M\"ob$^n$ - homogeneous tuples of operators in the Cowen-Douglas class $B_r(\D^n)$ is obtained in the paper \cite{FHOCD} for arbitrary $r$. Note that the transitivity of the action of the groups M\"ob$^n$ and Aut$(\D^n)$ on $\mathbb D^n$ plays an important role in the study of M\"ob$^n$ - homogeneous tuples as well as Aut$(\mathbb D^n)$ - homogeneous tuples in the Cowen-Douglas class $B_r(\D^n)$. The non-transitivity of the action M\"ob on $\mathbb D^2$ makes the study of M\"ob - homogeneous pairs in the Cowen - Douglas class $B_1(\mathbb D^2)$ very different.

This section begins with the observation that if $(M_{z_1}^*, M_{z_2}^*)$ is a M\"ob - homogeneous pair of operators on a reproducing kernel Hilbert space $(\mathcal{H}, K)$ in $B_1(\D^2)$, then there exists a multiplier representation $\pi : \mbox{M\"ob} \to \mathcal{U}(\mathcal{H})$ such that $\pi(\varphi) \varphi(M_{z_i}) = M_{z_i} \pi(\varphi)$, $i = 1, 2$, $\varphi \in$ M\"ob. Due to Theorem \ref{thm3}(i), the Hilbert space $\mathcal{H}$ then decomposes into a orthogonal direct sum $\mathcal{H} = \oplus_{n = 0}^\infty \mathcal{S}_n$ of irreducible subspaces of $\pi$. With respect to this decomposition of the Hilbert space $\mathcal{H}$, each of the operators $M_{z_i}^*$ has an upper triangular representation. Suppose $P_n$ is the orthogonal projection onto the subspace $\mathcal{S}_n$. It is also proved that there exists $\lambda > 0$ such that $P_n M_{z_i}^*|_{\mathcal{S}_n}$ is unitarily equivalent to ${M^{(\lambda + 2n)}}^*$ for each $n \geq 0$, where $M^{(\lambda + 2n)}$ is the multiplication by the coordinate function on $\mathbb A^{(\lambda + 2n)}(\D)$.

\begin{prop}\label{prop5}
Let $(M_1, M_2)$ be a pair of multiplication operators on a reproducing kernel Hilbert space $(\mathcal{H}, K)$ such that $(M_1^*, M_2^*) \in B_1(\D^2)$. If $(M_1, M_2)$ is homogeneous with respect to M\"ob, then there exists a multiplier representation $\pi : \mbox{M\"ob} \to \mathcal{U}(\mathcal{H})$ such that $\varphi(M_i)\pi(\varphi) = \pi(\varphi) M_i$, $i = 1, 2$, $\varphi \in$ M\"ob. 
\end{prop}

\begin{proof}
The existence of a unique (up to equivalence) projective unitary representation $\pi : \mbox{M\"ob} \to \mathcal{U}(\mathcal{H})$ satisfying 
\begin{equation}\label{neweq1}
\varphi(M_i)\pi(\varphi) = \pi(\varphi) M_i,\,\, i = 1, 2,\,\, \varphi \in \mbox{M\"ob},    
\end{equation}
follows along the line of the proof of \cite[Theorem 2.2]{THS}. Let $\varphi$ be an arbitrary element of M\"ob and $\boldsymbol{w}$ be an arbitrary element of $\mathbb D^2$. Due to the fact that $(M_1^*, M_2^*) \in B_1(\D^2)$, $\displaystyle \cap_{i=1}^2 \ker \left(M_i - \bar{\boldsymbol{w}}I\right)$ is the one dimensional subspace of $\mathcal{H}$ spanned by $\{K(\cdot, \boldsymbol{w})\}$. Since $(\varphi(M_1)^*, \varphi(M_2)^*)$ is unitarily equivalent to $(M_1^*, M_2^*)$, the subspace $\displaystyle \cap_{i=1}^2 \ker \left(\varphi(M_i) - \bar{\boldsymbol{w}}I\right)$ is also one dimensional and spanned by $\{K(\cdot, \varphi^{-1}(\boldsymbol{w}))\}$. Consequently, due to Equation \eqref{neweq1}, $\pi(\varphi)K(\cdot, \boldsymbol{w})$ lies in $\displaystyle \cap_{i=1}^2 \ker \left(\varphi(M_i) - \bar{\boldsymbol{w}}I\right)$ and therefore, we have
$$\pi(\varphi)K(\cdot, \boldsymbol{w}) = c(\varphi^{-1}, \boldsymbol{w}) K(\cdot, \varphi^{-1}(\boldsymbol{w}))$$
for some scalar $c(\varphi^{-1}, \boldsymbol{w})$. This proves that $\pi$ is a multiplier representation of M\"ob.
\end{proof}

Given a M\"ob - homogeneous pair of multiplication operators $(M_1, M_2)$ on a reproducing kernel Hilbert space $(\mathcal{H}, K)$ with $(M_1^*, M_2^*) \in B_1(\D^2)$, let $\mathcal{M}_n, \mathcal{S}_n, \mathcal{A}_n, P_n$ and $\Gamma_n$ be same as in the previous section. 
\begin{lem}\label{lem3}
For every $n \geq 0$, the subspace $\mathcal{M}_n$ is an invariant subspace of both the operators $M_{z_1}$ and $M_{z_2}$.
\end{lem}

\begin{proof}
Let $f \in \mathcal{M}_n$. Then we have $\partial_1^if|_{\Delta} = 0$ for $0 \leq i \leq n$. An easy computation shows that 
$$\partial_1^i \left(M_{z_1}f\right)(\bz) = i\partial_1^{i-1}f(\bz) + z_1\partial_1^if(\bz)$$
for every $\bz = (z_1, z_2) \in \D^2$. Since $\partial_1^if|_{\Delta} = 0$ for $0 \leq i \leq n$, it follows that $\partial_1^i M_{z_1}f|_{\Delta} = 0$ for $0 \leq i \leq n$. This proves that $\mathcal{M}_n$ is an invariant subspace of $M_{z_1}$. A similar computation shows that $\mathcal{M}_n$ is an invariant subspace of $M_{z_2}.$
\end{proof}

\begin{rem}\label{rem 3.x}
Note that $\oplus_{i = 0}^n\mathcal{S}_i = \mathcal{M}_n^\perp$ for every $n \geq 0$. The above lemma yields that $\oplus_{i = 0}^n\mathcal{S}_i$ is an invariant subspace of $M_1^*$ and $M_2^*$ for each $n \geq 0$. 
\end{rem}

\begin{prop}\label{prop3}
For every $n \geq 0$, let $M_n$ denote the multiplication operator on the reproducing kernel Hilbert space $\mathcal{A}_n$. Then $\Gamma_n \left(P_n M_{z_i}^*|_{\mathcal{S}_n}\right)\Gamma_n^* = M_n^*$ for $i = 1, 2$ and $n \geq 0$.
\end{prop}

\begin{proof}
For $\bw = (w_1, w_2) \in \D^2$, since $M_{z_1}^* K(\cdot, \bw) = \overbar{w}_1 K(\cdot, \bw)$, it follows from a direct computation that 
\begin{equation*}
M_{z_1}^* \overbar{\partial}_1^n K(\cdot, \bw) = n \overbar{\partial}_1^{n-1} K(\cdot, \bw) + \overbar{w}_1\overbar{\partial}_1^n K(\cdot, \bw)
\end{equation*}
for every $n \geq 0$. For $\bw \in \Delta$, since $\overbar{\partial}_1^{n-1}K(\cdot, \bw) \in \mathcal{M}_{n-1}^\perp$, it follows from the above equation that
\begin{equation}\label{eq:17}
P_nM_{z_1}^* \overbar{\partial}_1^n K(\cdot, \bw) = \overbar{w}P_n\overbar{\partial}_1^n K(\cdot, \bw),
\end{equation}
where $\bw=(w, w) \in \Delta$.
Let $P_n^\perp$ denote the orthogonal projection onto the space $\mathcal{S}_n^\perp$. We claim that for $\bw \in \Delta$, $P_n^\perp\overbar{\partial}_1^n K(\cdot, \bw) \in \mathcal{M}_{n-1}^\perp$ for every $n \geq 1$ and $P_0^\perp K(\cdot, \bw) = 0$. Indeed, since $K(\cdot, \bw) \in \mathcal{S}_0$ for every $w \in \Delta$, it follows that $P_0^\perp K(\cdot, \bw) = 0$ for every $\bw \in \Delta$. Now, assume that $n \geq 1$. Note that $\overbar{\partial}_1^n K(\cdot, \bw) \in \mathcal{M}_n^\perp$ for every $\bw \in \Delta$ and therefore, $P_n^\perp\overbar{\partial}_1^n K(\cdot, \bw) \in M_{n}^\perp \ominus \mathcal{S}_n = \mathcal{M}_{n-1}^\perp$ for every $\bw \in \Delta$.

It follows from Remark \ref{rem 3.x} that $\mathcal{M}_{n}^\perp$ is an invariant subspace of $M_{z_1}^*.$ Thus, we have $P_nM_{z_1}^* P_n^\perp \overbar{\partial}_1^n K(\cdot, \bw) = 0$ for every $\bw \in \Delta$ and therefore, from Equation \eqref{eq:17}, it follows that 
\begin{equation*}\label{eq:19}
P_nM_{z_1}^* P_n\overbar{\partial}_1^n K(\cdot, \bw) = \overbar{w}P_n\overbar{\partial}_1^n K(\cdot, \bw),\,\,\bw=(w, w) \in \Delta.
\end{equation*}
Now, Proposition \ref{prop2} yields that $\Gamma_n \left(P_n M_{z_1}^*|_{\mathcal{S}_n}\right)\Gamma_n^* = M_n^*$. Following a similar computation it is easy to show that $\Gamma_n \left(P_n M_{z_2}^*|_{\mathcal{S}_n}\right)\Gamma_n^* = M_n^*$.
\end{proof}

\begin{thm}\label{thm4}
Let the adjoint of a pair of multiplication operators $(M_{z_1}, M_{z_2})$ on a reproducing kernel Hilbert space $(\mathcal{H}, K)$ be in $B_1(\mathbb{D}^2)$. Assume that $\mathcal{H}$ contains all the polynomials. If $(M_{z_1}, M_{z_2})$ is M\"ob - homogeneous, then there exists a decomposition of the Hilbert space $\mathcal{H} := \oplus_{n = 0}^{\infty} \mathcal{S}_n$ and $\lambda > 0$ such that $P_nM_{z_i}^* P_n$ is unitarily equivalent to ${M^{(\lambda + 2n)}}^*$, $i = 1, 2,$ $n \geq 0$, where $P_n$ is the orthogonal projection onto $\mathcal{S}_n$ and $M^{(\lambda + 2n)}$ is the multiplication operator on the weighted Bergman space $\mathbb{A}^{(\lambda + 2n)}(\D)$.
\end{thm}

\begin{proof}
From Proposition \ref{prop5}, it follows that there exists a multiplier representation $\pi : \mbox{M\"ob} \to \mathcal{U}(\mathcal{H})$ such that $\pi(\varphi) \varphi(M_i) = M_i \pi(\varphi)$, $i = 1, 2$, $\varphi \in$ M\"ob. 

For every $n \geq 0$, let $\mathcal{M}_n$ be the subspace defined in Equation \eqref{eq:14} and $\mathcal{S}_n$ be the subspace defined in the line below Equation \eqref{eq:14}. Also, suppose $\Gamma_n$ is the map defined in Equation \eqref{eq:5} and $\mathcal{A}_n : =\mbox{Im}\,\,\Gamma_n$. Let $K_n$ be the reproducing kernel of $\mathcal{A}_n$.

From Proposition \ref{prop3}, it follows that $\Gamma_n \left(P_n M_{z_i}^*|_{\mathcal{S}_n}\right)\Gamma_n^* = M_n^*$ for $i = 1, 2$ and $n \geq 0$, where $M_n$ is the multiplication operator on $\mathcal{A}_n$. Again, from Theorem \ref{thm3}, we have a $\lambda > 0$ such that the reproducing kernel $K_n$ satisfies the Equation \eqref{eq:21}. This implies that $M_n^*$ is unitarily equivalent to $ {M^{(\lambda + 2n)}}^*$ for every $n \geq 0$. Therefore, it follows that $P_nM_{z_i}^* P_n$ is unitarily equivalent to ${M^{(\lambda + 2n)}}^*$, $i = 1, 2,$ $n \geq 0$. 
\end{proof}

\textit{Acknowledgement.} The authors are grateful to Professor Gadadhar Misra for his invaluable comments and suggestions in preparation of this article.

\end{document}